\documentclass[12pt]{article}
\usepackage{amsfonts,amsthm,amssymb,amsmath}

\usepackage{graphicx}

\newtheorem{definition}{Definition}

\newtheorem{proposition}{Proposition}
\newtheorem{remark}{Remark}
\newtheorem{theorem}{Theorem}

\title{Lagrange and Wolf Dualities in \\Nonholonomic Optimization}

\author{Constantin Udri\c ste, M\u ad\u alina Constantinescu, \\Ionel \c Tevy, Oltin Dogaru}
\date{}

\begin{document}

\maketitle

\begin{abstract}
This article deals with optimizing problems classified by the kinds of restrictions
as required in differential geometry and in mechanics: holonomic and nonholonomic. The central issue relates to
dual nonholonomic programs (what they mean and how they are solved?) when the nonholonomic
constraints are given by Pfaff equations. The original results are surprising and include
aspects derived from the Vranceanu theory of nonholonomic manifolds, from the geometric distributions theory and
from Darboux's theorem on canonical coordinates in which we can express a Pfaff form. On these ways we get also an
original Riemannian geometry attached to a given constrained optimization problem.
\end{abstract}



\section{Classical Lagrange and Wolfe \\dual programs}

In this Section, we re-discuss the well-known programs with
holonomic constraints insisting on the following issues \cite{B}, \cite{H}, \cite{R}:
(i) the Lagrange-dual problem with weak respectively strong duality;
(ii) the Wolfe-dual problem; (iii) pertinent examples.

\subsection{The Lagrange dual problem}

Let $D$ be a domain in $R^n$, let $x=(x^1,...,x^n)$ be a point in D and $f:D\to R$ and $g_\alpha: D\to R$ be convex functions.
Denote $g=(g_\alpha)$ and we introduce the set
$$\Omega = \{x\in D \,|\,g_\alpha(x)\leq 0,\,\alpha=1,...,m\}= \{x\in D \,|\,g(x)\preceq 0\}.$$
A complete notation for this set is $(\Omega, g, \preceq)$, but for short the sign $\preceq$
or the pair $(g,\preceq)$ are suppressed in the notation.
Let us consider the {\it convex program}
$$\min_x\{f(x)\,|\, x \in \Omega\}.\leqno(P)$$
The {\it Lagrange function (or Lagrangian)} of (P)
$$L(x, \lambda) = f(x)+ \sum_{\alpha=1}^m \lambda_\alpha g_\alpha(x)
= f(x)+<\lambda,g>, x\in D, \lambda\succeq 0$$
is convex in $x$ and linear in $\lambda$.

{\bf Remark} {\it We can create an original Riemannian geometry on the set of critical points,
using similar ideas we shall develop in a further 3.2.1}.

For all $\lambda\succeq 0$, the inequality
$$\sum_{\alpha =1}^m\lambda_\alpha g_\alpha(x)\leq 0, \,\,\forall x \in \Omega$$
holds. Consequently
$$L(x,\lambda)\leq f(x), \,\,\forall x \in \Omega,\,\, \forall \lambda\succeq 0.\leqno(1)$$
The equality holds iff (complementarity conditions)
$$\lambda_\alpha g_\alpha(x)=0 \,(\hbox{for each}\,\, \alpha =1,...,m).$$

Let us introduce the {\it Lagrange dual function}
$$\psi(\lambda)=\inf_{x\in D}\{f(x)+ \sum_{\alpha =1}^m \lambda_\alpha g_\alpha(x),\,\, x\in D,\,\, \lambda\succeq 0\}.$$
This function $\psi(\lambda)$ is concave, because it is a point-wise infimum of affine functions. Indeed,
using the linearity of $L(x,\lambda)$ with respect to $\lambda$, and introducing $\lambda^1\succeq 0$, $\lambda^2\succeq 0$,
and $0\leq t\leq 1$, we have
$$\psi(t\lambda^1 +(1-t)\lambda^2)=\inf_{x\in D}L(x,t\lambda^1 +(1-t)\lambda^2)$$
$$=\inf_{x\in D}(tL(x,\lambda^1) +(1-t)L(x,\lambda^2))\geq \inf_{x\in D}(tL(x,\lambda^1)) + \inf_{x\in D}((1-t)L(x,\lambda^2))$$
$$= t\inf_{x\in D}L(x,\lambda^1) + (1-t)\inf_{x\in D}L(x,\lambda^2)=t\psi(\lambda^1)+ (1-t)\psi(\lambda^2).$$

\begin{definition}
The problem
$$\sup_\lambda \,\{\psi(\lambda)\,|\, \lambda\succeq 0\}$$
is the so-called {\it Lagrange dual problem} of (P).
\end{definition}

The Lagrange dual problem can be called convex
because it is equivalent to the convex problem
$$\inf_\lambda\,\{-\psi(\lambda)\,|\, \lambda\succeq 0\}.$$

The Lagrange-dual problem is also defined
in this way if (P) is not convex. The following theorem holds also in that case.

\begin{theorem} ({\bf weak duality}) The dual function yields lower bounds of the initial optimal value $f_*$, i.e., for any
$\lambda$, we have $\psi(\lambda)\leq f_*$. In other words,
$$\sup_\lambda\,\{\psi(\lambda)\,|\, \lambda\succeq 0\}\leq \inf_{x\in \Omega}\{f(x)\}.$$
\end{theorem}

{\bf Proof} In the foregoing statements, we have the relation (1).
Since $\Omega\subset D$, for each $\lambda\succeq 0$, we find
$$\psi(\lambda)=\inf_{x\in D}L(x,\lambda)\leq \inf_{x\in \Omega}L(x,\lambda)\leq \inf_{x\in \Omega}f(x).$$
Thus the statement in the theorem is true.

The problem of finding the best lower bound on $f_*$ obtained from the Lagrange dual function is called the {\it
Lagrange dual problem} for the original or primal problem.

The optimal values may be different. However, they are equal if (P) satisfies the Slater
condition and has finite optimal value. This is the next result.

\begin{theorem} {\bf (strong duality)} If the program (P) satisfies the Slater condition and has finite optimal
value, then
$$\sup_\lambda\, \{\psi(\lambda)\,|\, \lambda\succeq 0\}= \inf_{x\in D}\{f(x)\,|\,\,g(x)\preceq 0\}.$$
Moreover, then the dual optimal value is attained.
\end{theorem}

{\bf Proof} Denote by $f_*$ the optimal value of (P). Taking $a = f_*$ in the Convex Farkas Lemma,
it follows that there exists a vector $\lambda_*=(\lambda_{1*},...,\lambda_{m*})\geq 0$ such that
$$L(x, \lambda_*) = f(x)+ \sum_{\alpha =1}^m \lambda_{\alpha *}\, g_\alpha(x)\geq f_*,\, \forall x \in D.$$
Using the definition of $\psi(\lambda_*)$ this implies $\psi(\lambda_*)\geq f_*$. By the weak duality theorem, it
follows that $\psi(\lambda_*)= f_*$. This not only proves that the optimal values are equal, but also
that $\lambda_*$ is an optimal solution of the dual problem.

{\bf Remark} Unlike in Linear Programming theory, the strong duality theorem cannot always be established for
general optimization problems.

\subsection{Topology of Lagrange multipliers}

Let $f:D\subseteq R^{n}\rightarrow R$ and $g:D\subseteq R^{n}\rightarrow
R^{p}$, $p<n,$ of class $C^{2}$ with $rank\,J_g = p$ in $D$. Let
$L\left( x,\lambda \right) =f\left( x\right) +\lambda \cdot g\left( x\right),$ with $%
\lambda \in R^{p}.$ We recall that
\begin{equation*}
H\left( x,\lambda \right)=\nabla f\left( x\right) +\lambda\cdot \nabla g\left( x\right) =0
\end{equation*}%
is the equation of critical points with respect to $x$ of the Lagrange function.

Let $A=\left\{ x\ |\ \exists \lambda \text{ cu }H\left( x,\lambda \right)
=0\right\} $ and $B=\left\{ \lambda \ |\ \exists x\text{ cu }H\left(
x,\lambda \right) =0\right\} .$ Introduce $h:A\rightarrow B$ such that
$H\left( x,h\left( x\right) \right) =0.$ The function $h$ is well defined
since the equation $H\left( x,\lambda \right) =0$ is linear in $\lambda$ (system with unique solution).
Hence, for any $\lambda \in B$, the set $h^{-1}\left( \lambda \right)
$ is non-void, and it consists of all critical points corresponding to $\lambda$
(set in which the nondegenerate critical points are isolated).

\begin{proposition}
Let $\lambda _{0}\in B$ such that there exists $x_{0}\in
h^{-1}\left( \lambda _{0}\right) $ with the property that  $x_{0}$ is
nondegenerate, i.e., the Hessian $d^{2}f\left( x_{0}\right) +\lambda _{0}\cdot d^{2}g\left(
x_{0}\right) $ is nondegenerate. Then $h$ admits a differentiable section $s_{\lambda _{0}}:I_{\lambda _{0}}\rightarrow A.$
\end{proposition}

\begin{proof}
Since $\frac{\partial H}{\partial x}\left( x_{0},\lambda _{0}\right)
=d^{2}f\left( x_{0}\right) +\lambda _{0}\cdot d^{2}g\left( x_{0}\right)$ is non-degenerate, by hypothesis,
there exists a neighborhood $I_{\lambda _{0}}$ of $%
\lambda _{0}$ and a differentiable function $s_{\lambda
_{0}}:I_{\lambda _{0}}\rightarrow A$ such that $H\left(
s_{\lambda _{0}}\left( \lambda \right) ,\lambda \right) =0,$ $\forall
\lambda \in I_{\lambda _{0}}$ and $s_{\lambda _{0}}\left( \lambda
_{0}\right) =\lambda _{0}.$ Moreover, the function $s_{\lambda _{0}}$ is unique,
with these properties.
\end{proof}

For any $\lambda \in B$, let $S_{\lambda }$ be the set of all sections of $h$ defined
in a neighborhood of $\lambda,$ set which is eventually void.

\begin{remark}
(i) If  $h^{-1}\left( \lambda \right) $ contains at least one nondgenerate critical point,
then $S_{\lambda }$ is non-void. If $%
h^{-1}\left( \lambda \right) $ does not contain degenerate critical points,
then the sets $h^{-1}\left( \lambda \right) $ and $S_{\lambda }$
have the same cardinal and are discrete sets.

(ii) The set $C=\left\{ \lambda \in B\ |\text{ }S_{\lambda }\neq
\emptyset \right\} $ is open.
\end{remark}

In the following, we suppose that the set $S_{\lambda }$
is finite, for any $\lambda \in B.$ We can define $f^{\ast
}:B\rightarrow R$ by $f^{\ast }\left( \lambda \right) =\max\nolimits_{s\in
S_{\lambda }}f\left( s\left( \lambda \right) \right) ,$ if $S_{\lambda
}\neq \emptyset $ and $f^{\ast }\left( \lambda \right) =-\infty ,$ if $S_{\lambda }=\emptyset .$

\begin{proposition}
(i) For any $\lambda \in B$, we have
\begin{equation*}
f^{\ast }\left( \lambda \right) \leq \sup\nolimits_{x\in h^{-1}\left(
\lambda \right) }f\left( x\right) .
\end{equation*}

(ii) If $h^{-1}\left( \lambda \right) $ does not contain degenerate critical points, then
\begin{equation*}
f^{\ast }\left( \lambda \right) =\sup\nolimits_{x\in h^{-1}\left( \lambda
\right) }f\left( x\right) =\max\nolimits_{x\in h^{-1}\left( \lambda \right)
}f\left( x\right) .
\end{equation*}
\end{proposition}

\begin{proof}
(i) Let $s_{0}\in S_{\lambda }$ cu $f\left( s_{0}\left( \lambda \right)
\right) =\max\nolimits_{s\in S_{\lambda }}f\left( s\left( \lambda \right)
\right) =f^{\ast }\left( \lambda \right) .$ Since $s_{0}\left( \lambda
\right) \in h^{-1}\left( \lambda \right) ,$ it follows that $f^{\ast
}\left( \lambda \right) \leq \sup\nolimits_{x\in h^{-1}\left( \lambda
\right) }f\left( x\right) .$

(2) By hypothesis, the sets
$S_{\lambda }$ \c{s}i $h^{-1}\left( \lambda \right) $ have the same cardinal, hence
 $h^{-1}\left( \lambda \right) $ is finite. Let $y_{0}\in h^{-1}\left( \lambda \right) $ with $f\left( y_{0}\right) =\max_{x\in
h^{-1}\left( \lambda \right) }f\left( x\right) .$ Since $\left(y_{0},\lambda \right) $ is a nondegenerate critical point, there exists $%
s_{1}\in S_{\lambda }$ with $s_{1}\left( \lambda \right) =y_{0}.$ Then, it follows
that $$\max\limits_{s\in S_{\lambda }}f\left( s\left( \lambda
\right) \right) \geq f\left( y_{0}\right) =\max_{x\in h^{-1}}\left( \lambda
\right)f\left( x\right) .$$
\end{proof}

\begin{proposition}
Let $\lambda _{0}\in B$ such that $h^{-1}\left( \lambda
_{0}\right) $ does not contain degenerate critical points. Suppose, also,
that $\ f|_{h^{-1}\left( \lambda _{0}\right) }$ is injective. Then there exists $s_{0}\in S_{\lambda _{0}},$ $s_{0}:I_{0}%
\rightarrow A$ such that $f^{\ast }\left( \lambda \right)
=f\left( s_{0}\left( \lambda \right) \right) ,$ for any $\lambda \in
I_{0}.$
\end{proposition}

\begin{proof}
Let $s_{0}\in S_{\lambda _{0}},$ $s_{0}:I_{0}\rightarrow A$ such that
$f^{\ast }\left( \lambda _{0}\right) =f\left( s_{0}\left( \lambda
_{0}\right) \right).$ Then $f^{\ast }\left( \lambda _{0}\right) =f\left(
s_{0}\left( \lambda _{0}\right) \right) >f\left( s\left( \lambda _{0}\right)
\right) ,$ $\forall s\in S_{\lambda _{0}},\ s:I_{s}\rightarrow A.$ Since
$f$ is continuous and the set $S_{\lambda _{0}}$ is finite,
it follows that we can restrict the neighborhood $I_{0}$ such that
$f\left( s_{0}\left( \lambda \right) \right) >f\left( s\left(
\lambda \right) \right) ,$ $\forall \lambda \in I_{0},\forall s\in
S_{\lambda _{0}},$ i.e., $f^{\ast }\left( \lambda \right) =f\left(
s_{0}\left( \lambda \right) \right) ,\forall \lambda \in S_{\lambda _{0}}.$
\end{proof}

\subsection {The meaning of Lagrange multiplier}

In our mostly geometrical discussion, $\lambda$ is just an artificial variable that lets us compare the
directions of the gradients without worrying about their magnitudes. To express mathematically the meaning of the multiplier,
write the constraint in the form $g(x) = c$ for some constant $c$.
This is mathematically equivalent to our usual $g(x)=0$, but allows us to easily describe a whole family of constraints.
For any given value of $c$, we can use Lagrange multipliers to find the optimal value of $f(x)$ and the point where it occurs.
Call that optimal value $f_*$, occurring at coordinates $x_0$ and with Lagrange multiplier  $\lambda _0$.
The answers we get will all depend on what value we used for $c$ in the constraint, so we can think of these as functions of $c$ :
$ f_*(c), x_0(c), \lambda_0(c)$ .
Of course, $ f(x)$ only depends on $c$ because the optimal coordinates $x_0$ depend on $c$: we could write it as $f_*(c)$.

To find how the optimal value changes when we change the constraint, just take the derivative
$$\frac {df_*}{dc} = \frac {\partial f_*}{\partial x^i_0}\, \frac {dx^i_0}{dc} = \nabla f_* \,\cdot \, \frac {dx_0}{dc}.$$
Use the equation of critical points to substitute $\nabla f_*  = - \lambda_0 \, \nabla g_0$ and obtain
$$\frac {df_*}{dc} = - \lambda_0 \, \nabla g_0 \,\cdot \, \frac {dx_0}{dc} =  - \lambda_0 \,\frac {dg_0}{dc}\,.$$
But the constraint function $g_0 = g(x_0(c))$ is {\it always}  equal to $c$, so $dg_0/dc = 1$. Thus, $df_*/dc = - \lambda _0$.
That is, the Lagrange multiplier is the {\it rate of change of the optimal value with respect to changes in the constraint}.

Of course, $f_*$ depends on $c$ through of $\lambda$, and then ${\displaystyle \frac {df_*}{dc} =  \frac {df_*}{d \lambda}\,
\frac {d\lambda}{dc} }$. We can define $c(\lambda ) $ by Cauchy problem
$$ \frac {dc}{d \lambda } = - \frac {1}{\lambda }\, \frac {df_*}{d \lambda}\,,\,\,\,\,\,\, c(\lambda _0) = 0\,. \leqno (EC)$$
Then another Lagrange dual function may be
$$  \varphi (\lambda ) = f_*( x_0(\lambda ) ) + \lambda \, c(\lambda )\,. \leqno (LDF)$$

{\bf Proposition} {\it If optimum points are critical points, both Lagrange dual functions give the same solution. Hence strong duality holds.}

{\bf Proof} Indeed, using (EC) we have
$$  \varphi ^\prime (\lambda ) =  \frac {df_*}{d \lambda} + c(\lambda ) + \lambda \, \frac {dc}{d \lambda } = c(\lambda )\, $$
and $  \varphi ^\prime (\lambda ) = 0$ implies $ c(\lambda ) =0$, that is for $\lambda _0$.

Often the Lagrange multiplier have
an interpretation as some quantity of interest:

(i) $\lambda$ is the rate of change of the quantity being optimized as a function of the constraint variable
since $\frac{\partial L}{\partial c}=\lambda$;

(ii) by the envelope theorem the optimal value of a Lagrange multiplier has an interpretation
as the marginal effect of the corresponding constraint constant upon the optimal attainable
value of the original objective function: if we denote values at the optimum with an asterisk, then it can be shown that
$$ \frac{d}{dc}\, f_*=  \frac{d}{dc} f(x(c)) = \lambda_*.$$

For details regarding classical theory of programs see \cite{B}, \cite{H}, \cite{R}.

If we have more constraints $g_\alpha (x) = c_\alpha \,, \alpha = 1, ..., m$, then the Lagrange function is $L(x, \lambda) = f(x) +
\sum_{\alpha=1}^m \lambda _\alpha (g_\alpha (x) -c_\alpha)$ and the system of critical points is
$$ \frac {\partial f}{\partial x^i} + \sum_{\alpha=1}^m \lambda_\alpha \,\frac {\partial g_\alpha}{\partial x^i} = 0\,.$$

Because the optimal coordinates $x_0$ and the optimal value $f_*$ depend on vector $c$, taking the derivatives we have
$$\frac { \partial f_*}{\partial c_\alpha} = \frac {\partial f_*}{\partial x^i_0}\, \frac {\partial x^i_0}{\partial c_\alpha} =
- \sum_{\beta=1}^m \lambda_\beta^0 \, \frac {\partial g_\beta}{\partial x^i_0}\, \frac {\partial x^i_0}{\partial c_\alpha}  =
- \sum_{\beta=1}^m \lambda_\beta^0\, \frac {\partial g_\beta}{\partial c_\alpha} = - \sum_{\beta=1}^m \lambda_\beta^0\, \delta_{\alpha\beta} = - \lambda_\alpha^0\,.$$
Then we can define $c(\lambda ) $ by the partial differential system, with initial condition, written in matrix language as
$$[\lambda_1 \,\,...\,\, \lambda_m]\,
\left[\begin{array}{ccc}\frac{\partial c_1}{\partial \lambda_1}&...&\frac{\partial c_1}{\partial \lambda_m}\\ \
...&...&...\\ \
\frac{\partial c_{m}}{\partial \lambda_1}&...&\frac{\partial c_{m}}{\partial \lambda_m}
\end{array}\right] =
- \left [ \frac {\partial f_*}{\partial \lambda_1}\,\, ...\,\, \frac {\partial f_*}{\partial \lambda_m} \right ] \,,\,\,\,\,\,\, c(\lambda _0) = 0\,. \leqno (EC)$$
Then another Lagrange dual function may be
$$  \varphi (\lambda ) = f_*( x_0(\lambda ) ) + \sum_{\alpha=1}^m \lambda_\alpha\, c_\alpha (\lambda )\,. \leqno (LDF)$$

Deriving the $\beta$-th equation with respect to $\lambda_\alpha$ and the $\alpha$-th equation with respect to $\lambda_\beta$ , in the previous
system, we obtain the complete integrability conditions ${\displaystyle \frac {\partial c_\alpha}{\partial \lambda_\beta} =
\frac {\partial c_\beta}{\partial \lambda_\alpha}}\,$ (symmetric Jacobian matrix); consequently $c(\lambda )$ is
the gradient of a scalar function, namely the Lagrange dual function
$\varphi (\lambda )$.

Moreover, the previous square matrix being a symmetrical one, we can write the equation $(EC)$ as
$$[\lambda ^1 \,\,...\,\, \lambda ^m]\,
\left[\begin{array}{ccc}\displaystyle\frac{\partial c_1}{\partial \lambda^1}&...&\displaystyle\frac{\partial c_m}{\partial \lambda^1}\\ \
...&...&...\\ \
\displaystyle\frac{\partial c_1}{\partial \lambda^m}&...&\displaystyle\frac{\partial c_{m}}{\partial \lambda^m}
\end{array}\right] =
- \left [ \frac {\partial f_*}{\partial \lambda ^1}\,\, ...\,\, \frac {\partial f_*}{\partial \lambda ^m} \right ] \,.$$

Consequently, in a regular case, we have the following situation: Solving a constrained optimum problem we obtain the optimal value as
$f_* = f(c_1, ... , c_m)$. For the dual problem we use a $f_* = f( \lambda^1, ... ,  \lambda^m)$. If the correspondence between $(c_1, ... , c_m)$ and
$( \lambda^1, ... ,  \lambda^m)$ is like a change of variables there hold the relations:
$$grad_c\, f_* = - \lambda\,; \,\,\,\,\, c ^\prime (\lambda )\,  \lambda  = - grad_\lambda \, f_* \,,\,\, c ^\prime (\lambda ) \in L(R^m,\, R^m)\,.$$

\subsection{Examples and counter-examples}

(1) Let us consider the functions $f(x,y)=x^2+y^2$ and $g(x,y)=x^2+y^2-2x$ and the problem
$$\min f(x,y) \,\,\hbox{constrained by}\,\, g(x,y)=c, \,c\geq -1.$$
The Lagrange function of this problem is
$$L(x,y,\lambda)= x^2+y^2+\lambda (x^2+y^2-2x - c).$$
The critical points of the partial function $(x,y)\to L(x,y,\lambda)$ are the solutions of the system
$$\frac{1}{2}\frac{\partial L}{\partial x}= x+\lambda x -\lambda=0,\,\frac{1}{2}\frac{\partial L}{\partial y}= y+\lambda y=0.$$
It follows $x=\frac{\lambda}{\lambda+1},\,y=0$ and hence $f_* = \left(\frac{\lambda}{\lambda +1}\right)^2$.
On the other hand, by restriction, in critical points, we have the relation
$$\left(\frac{\lambda}{\lambda +1}\right)^2 - 2 \frac{\lambda}{\lambda +1}=c.$$
It follows
$$\frac{df_*}{dc}= \frac{df_*}{d\lambda}\,\,\frac{1}{\frac{dc}{d\lambda}}$$
and finally, we obtain the geometrical interpretation $\frac{df_*}{dc}= -\lambda$.

The dual function is
$$\psi(\lambda)= L(x(\lambda),y(\lambda),\lambda)= - \frac{\lambda^2}{\lambda+1}-\lambda c.$$
The value $\psi(\lambda)$ is a minimum for $\lambda >-1$ and a maximum for $\lambda<-1$, in the initial problem.
The condition of extremum (critical point), $\psi^\prime(\lambda)=0$, is equivalent to
$(c+1)(\lambda+1)^2=1$ and the dual problem has the same solution as the primal one.

\begin{figure}
  \centering
 \includegraphics[width={10cm}]{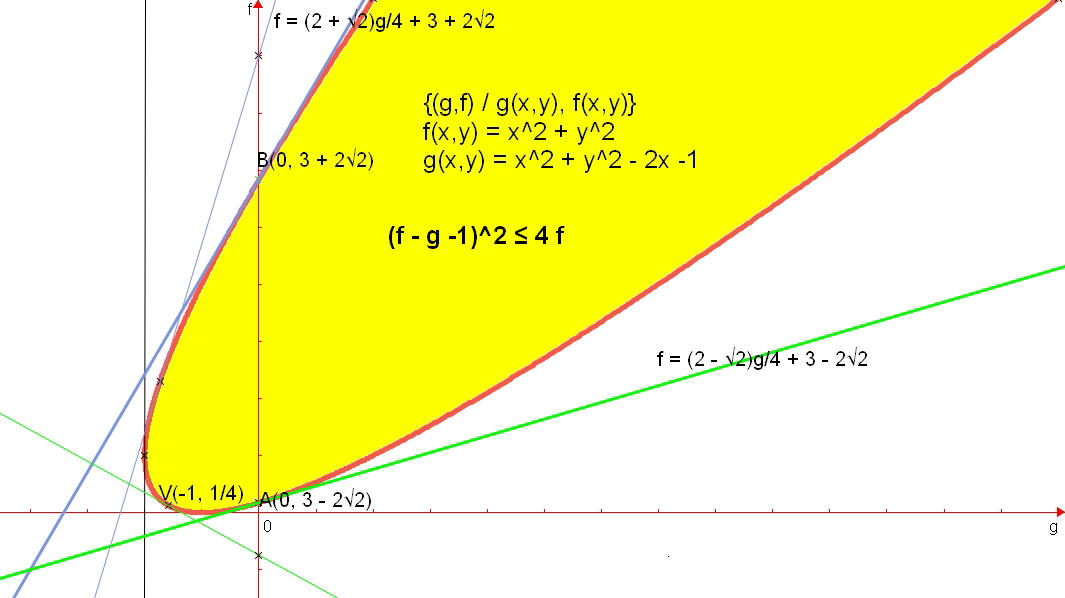}\\
  \caption{Geometry of Lagrange duality}\label{Fig.1}
\end{figure}

On the other hand the equation $(EC)$ for this problem is
$$ \frac {d\bf{c}}{d \lambda } = - \frac {1}{\lambda }\, \frac {d}{d \lambda}\,\left (\frac {\lambda}{\lambda + 1} \right )^2,\,\,\,\,\,\,
{\bf c} (\lambda _0) = 0\,,\,\, \hbox {where}\,\, \lambda _0 +1 = \pm \frac {1}{ \sqrt {c + 1}}\,.$$
We find $\displaystyle {\bf c}(\lambda ) = \frac {1}{(\lambda + 1)^2} - 1 - c $ and the dual Lagrange function
$$\varphi (\lambda ) = \left ( \frac {\lambda }{\lambda + 1} \right )^2 + \lambda \,\left (\frac {1}{(\lambda + 1)^2} - 1 - c \right ) = \psi(\lambda) $$
as the above one.

(2) {\bf A problem with two constraints} Solve the following constrained optimum problem: $$f(x,y,z) = xyz = extremum $$
{\it constrained by}
$$g_1(x,y,z)=x+y -a=0,\,\,g_2(x,y,z)=xz+yz-b=0\,.$$
The Lagrange function is
$$L(x,y,z, \lambda , \mu ) = xyz + \lambda (x+y-a) + \mu (xz+yz-b) $$
and the feasible solution of the problem is, only,
$$x=y=\frac {a}{2}= -2 \mu \,,\, z=\frac {b}{a}= \frac {\lambda}{\mu}\,,\, \lambda = - \frac {b}{4}\,,\, \mu = - \frac {a}{4}\,, $$
$$f_* = \frac {ab}{4} = 4 \lambda \mu \,.$$
The Lagrange dual function is $\psi(\lambda , \mu) = -4 \lambda \mu -a \lambda - b \mu \,.$

The partial differential system which defines $c_1(\lambda , \mu )$ and $c_2(\lambda , \mu )$ becomes, in this case,
$$\left [\lambda \,\,\,\, \mu \right ]\,
\left[\begin{array}{cc} \displaystyle {\frac{\partial c_1}{\partial \lambda}}& \displaystyle { \frac{\partial c_1}{\partial \mu}}\\
\displaystyle {\frac{\partial c_{2}}{\partial \lambda}} & \displaystyle {\frac{\partial c_{2}}{\partial \mu }}
\end{array}\right] = -  [ 4 \mu \,\,\,\,  4 \lambda ] \,,\,\,\, c_1\left (- \frac {b}{4},  - \frac {a}{4} \right ) =
c_2\left (- \frac {b}{4},  - \frac {a}{4} \right ) = 0\,. \leqno (EC)$$
Taking into account that $\displaystyle \frac{\partial c_1}{\partial \mu} = \frac{\partial c_{2}}{\partial \lambda} $,
we obtain two quasilinear PDEs
$$\lambda \frac{\partial c_1}{\partial \lambda } + \mu \frac{\partial c_1}{\partial \mu} = -4 \mu,\,\,
\lambda \frac{\partial c_2}{\partial \lambda } + \mu \frac{\partial c_2}{\partial \mu} = -4 \lambda \,, $$
with solutions, respectively
$$c_1(\lambda, \mu) = -4 \mu + \alpha \left (\frac {\lambda}{\mu} \right ), \, \, c_2(\lambda, \mu) = -4 \lambda + \beta \left (\frac {\lambda}{\mu} \right ),\,$$
$\alpha \,, \beta $ arbitrary functions.
The condition $\displaystyle \frac{\partial c_1}{\partial \mu} = \frac{\partial c_{2}}{\partial \lambda} $
is verified, for instance, if $\alpha $ and $\beta $ are constant functions.
Using the initial conditions, we find finally
$$c_1(\lambda, \mu) = -4 \mu - a \,, \, c_2(\lambda, \mu) = -4 \lambda - b\,,$$
$$\varphi (\lambda , \mu ) = 4 \lambda \mu + \lambda (-4 \mu - a) + \mu (-4 \lambda - b) = \psi(\lambda , \mu)\,. $$

The Geometry of Lagrange duality is represented in Figure 1.

(3) {\bf A strange problem} Solve the following constrained optimum problem: $$f(x,y,z) = xyz = extremum $$
{\it constrained by}
$$g_1(x,y,z)=x+y+z-a=0\,,$$ $$g_2(x,y,z)=xy+xz+yz-b=0\,.$$
So the Lagrange function is
$$L(x,y,z, \lambda , \mu ) = xyz + \lambda (x+y+z-a) + \mu (xy+xz+yz-b) $$
and one from the solutions of the problem is, for instance,
$$\mu = \frac {-a- \sqrt {a^2-3b}}{3}\,,\,\,\, \lambda = \frac {2a^2-3b+2a \sqrt {a^2-3b}}{9} = \mu ^2\,,$$
$$x=y= \frac {a+ \sqrt {a^2-3b}}{3} = - \mu \,,\,\,\,\, z= \frac {a-2 \sqrt {a^2-3b}}{3} = a + 2 \mu \,,$$
with the extremum value
$$f_*= \frac {1}{27} \left ( -2a^3 + 9ab - 2(a^2 -3b)^{3/2} \right ) = \lambda (a + 2 \mu )\,,$$
only if $a^2-3b \geq 0$.

{\bf Remark} Another solution of the problem is $$\mu = \frac {-a + \sqrt {a^2-3b}}{3}\,, ... \,\hbox {and so on} $$
with the extremum value
$$f^*= \frac {1}{27} \left ( -2a^3 + 9ab + 2(a^2 -3b)^{3/2} \right ) = \lambda (a + 2 \mu )\,.$$

The interval $[f_*\,, f^*]$ solves the following algebraic problem: {\it Find the real numbers $m$ such that the equation
$t^3 - at^2 + bt - m = 0, \,\, a, b \in R$, has tree real roots.}

It is easily to verify that
$$\frac {\partial }{\partial a} f_*(a,b) = - \lambda \,\, \hbox {and} \,\, \frac {\partial }{\partial b} f_*(a,b) = - \mu \,.$$

On the other hand, $\displaystyle \frac {\partial (\lambda , \mu )}{\partial (a, b)} = 0 $ and $f_*$ cannot be expressed as function
of $\lambda $ and $\mu $ only. Then
we have to consider $f_* = f_*(a,b, \lambda (a,b), \mu (a,b))$  and the following relations
$$ - \lambda = D_a f_* = \frac {\partial f_*}{\partial a} +  \frac {\partial f_*}{\partial \lambda}  \frac {\partial \lambda}{\partial a} +
\frac {\partial f_*}{\partial \mu}  \frac {\partial \mu}{\partial a}$$
$$ - \mu = D_b f_* = \frac {\partial f_*}{\partial b} +  \frac {\partial f_*}{\partial \lambda}  \frac {\partial \lambda}{\partial b} +
\frac {\partial f_*}{\partial \mu}  \frac {\partial \mu}{\partial b}$$
which is easily to verify also (here $D_.$ is an operator of total derivative.)

{\bf Question } Which is the dual Lagrange function $\psi (\lambda, \mu )$ in this case?

Solving the system of the critical points with respect to $x, y$ and $z$ we find, for instance, $x=y= -\mu $, $z$ undeterminate and
$\lambda = \mu ^2$. With these, one obtains the dual Lagrange function
$$\psi (\lambda, \mu )=\chi(\mu) = - \mu ^3 - a \mu ^2 - b \mu \,.$$

{\bf Remark} Although $z$ is undeterminate, the dual Lagrange function does not depend upon $z$, because with the above solutions
the derivative $\displaystyle {\frac {\partial L}{\partial z}}$ vanishes identically. The critical points condition for the dual Lagrange function
$$\frac {d \chi }{d \mu } = - (3 \mu ^2 + 2a \mu + b) = 0$$
gives us the same solutions as in primal problem.

{\bf Open problem} How it means and how we find the functions $c_1$ and $c_2$ in the situation, like this, when
$\displaystyle \frac {\partial (\lambda , \mu )}{\partial (a, b)} = 0 $ ?



(4) Let us consider the functions $f(x,y)=x^2+y^2$ and $g(x,y)=x + y$, with $(x,y)\in R^2$, and the problem
$$\min f(x,y) \,\,\hbox{constrained by}\,\, g(x,y)\geq 1.$$
The Lagrange function is
$$L(x,y,\lambda)= x^2+y^2+\lambda (1- x - y),\,\,(x,y)\in R^2,\, \lambda \geq 0.$$
The function $(x,y)\to L(x,y,\lambda)$ is convex. Consequently, it is minimal iff
$$\frac{\partial L}{\partial x}=0, \frac{\partial L}{\partial y}=0.$$
This holds if $x = \frac{\lambda}{2},\,y = \frac{\lambda}{2}$.
Substitution gives
$$\psi(\lambda)=\lambda - \frac{\lambda^2}{2},\, \lambda \geq 0.$$
The dual problem $\max\psi(\lambda)$ has the optimal point $\lambda =1$.
Consequently, $x = y = \frac{1}{2}$ is the optimal solution of the original (primal)
problem. In both cases the optimal value equals $\frac{1}{2}$, i.e.,
at optimality the duality gap is zero!

(5) Let us solve the program
$$\min x \,\,\hbox{subject to}\,\,x^2\leq 0,\, x \in R.$$
This program is not Slater regular. On the other hand, we have
$$\psi(\lambda)=\inf_{x\in R}(x+ \lambda x^2)=\left\{\begin{array}{ccc} -\frac{1}{2\lambda} & for & \lambda >0\\ \
-\infty& for & \lambda =0.\end{array}\right.$$
Obviously, $\psi(\lambda)<0$ for all $\lambda \geq 0$. Consequently,
$\sup \{\psi(\lambda)\,|\,\lambda \geq 0\} = 0$.
So the Lagrange-dual has the same optimal value as the primal problem. In spite of the lack
of Slater regularity there is no duality gap.

(6) {\bf (Example with positive duality gap)}
We consider the program
$$\min e^{-y}\, \,\,\hbox{subject to}\,\,\,\sqrt{x^2+y^2}- x \leq 0,\, (x,y) \in R^2.$$
Here the feasible region is $\Omega =\{(x,y)\in R^2\,|\, x\geq 0, y=0\}$.
Consequently this program is not Slater regular. The optimal value of the objective function is $1$.
The Lagrange function is
$$L(x,y,\lambda)= e^{-y} + \lambda(\sqrt{x^2+y^2}- x).$$
The Lagrange dual program can be written in the form
$$\sup \psi(\lambda)\, \,\,\hbox{subject to}\,\,\,\lambda \geq 0.$$
Note that $L(x, y,\lambda) > 0$ implies $\psi(\lambda)\geq 0$. Now let $\epsilon >0$. Fixing $y = - \ln \epsilon$ and
$x= \frac{y^2-\epsilon^2}{2\epsilon}$, we find $\sqrt{x^2+y^2}- x=\epsilon$
and
$$L(x,y,\lambda)= (1+\lambda)\epsilon.$$
In this way,
$$\psi(\lambda)=\inf_{(x,y)\in R^2}L(x,y,\lambda)\leq \inf_{\epsilon>0}(1+\lambda)\epsilon=0.$$
On the other hand, we also have $\psi(\lambda)\geq 0$, and consequently the optimal value of the Lagrange dual
program is $0$, and hence the minimal duality gap equals $1$! ( No strong duality here).

\subsection{The Wolfe-dual problem}

The {\it Lagrange dual program} can be written in the form
$$\sup_{\lambda \geq 0}\,\,\{\,\inf_{x \in D}\{f(x) + \sum_{\alpha=1}^m \lambda_\alpha g_\alpha(x)\}\}.$$

Assume that $D = R^n$ and the functions $f, g_1, ... , g_m$ are continuously differentiable and
convex. For a given $\lambda \geq 0$ the inner minimization problem is convex, and we can use the
fact that the infimum is attained if and only if the gradient with respect to $x$ is zero.

\begin{definition}
The problem
$$\sup_{x,\lambda}\,\,\{f(x) + \sum_{\alpha=1}^m \lambda_\alpha g_\alpha(x)\} \leqno(WP)$$
subject to
$$\frac{\partial f}{\partial x^i}(x) + \sum_{\alpha=1}^m \lambda_\alpha \frac{\partial g_\alpha}{\partial x^i}(x)=0,\,\, \lambda \geq 0$$
is called the {\it Wolfe dual} of the program (P).
\end{definition}

Obviously, the constraints in Wolfe dual are usually nonlinear. In such cases the
Wolfe-dual is not a convex program.

The Wolfe dual has the weak duality property.

\begin{theorem} {\bf (weak duality property)}
Suppose that $D= R^n$ and the functions $f, g_1, ... , g_m$ are continuously differentiable and
convex. If $\hat x$ is a feasible solution of (P) and $(\bar x, \bar \lambda)$ is a feasible solution for (WP), then
$$L(\bar x, \bar \lambda) \leq f(\hat x).$$
In other words, weak duality holds for (P) and (WP).
\end{theorem}

\subsection{Example}

(1) Let us consider the convex program
$$\min_{x,y}\, x+ e^{y}\, \,\,\hbox{subject to}\,\,\,3x-2 e^y \geq 10,\, y\geq 0,\,\,(x,y) \in R^2.$$
Then the optimal value is $5$ with $x = 4, y=0$. The Wolfe dual of this program is
$$\sup_{x,y,\lambda}\{x+ e^{y} + \lambda_1 (10- 3x + 2 e^y)-\lambda_2 y\}$$
subject to
$$1-3\lambda_1=0,\,e^{y} + 2 e^{y}\lambda_1-\lambda_2=0,\,\,(x,y)\in R^2,\, \lambda \geq 0.$$
Obviously, the Wolfe dual program is not convex. It follows
$\lambda_1=\frac{1}{3}$ and the second constraint becomes $\frac{5}{3}e^{y}-\lambda_2=0$.
Eliminating $\lambda_1$, $\lambda_2$ from the objective function, we find
$$g(y)=\frac{5}{3}\,e^{y}- \frac{5}{3}\,y\,e^{y}+\frac{10}{3}.$$
This function has a maximum when $g^\prime(y)=0$, i.e., $y=0$ and $f(0)=5$.
Hence the optimal value of (WP) is $5$
and then $(x, y,\lambda_1,\lambda_2) = (4, 0, \frac{1}{3}, \frac{5}{3})$.

{\bf Remark} The substitution $z = e^y \geq 1$ makes the problem linear.

\section{Minimax inequality}

For any function $\phi$ of two vector variables $x\in X,\, y\in Y$,
the {\it minimax inequality}
$$\max_{y\in Y}\, \min_{x\in X}\, \phi(x,y)\leq \min_{x\in X}\, \max_{y\in Y} \,\phi(x,y)$$
is true. Indeed, start from
$$\forall x, y: \min_{x^\prime\in X}\phi(x^\prime, y)\leq \max_{y^\prime\in Y} \phi(x, y^\prime)$$
and take the minimum over $x \in X$ on the right-hand side, then the maximum over $y \in Y$
on the left-hand side.

Weak duality is a direct consequence of the minimax inequality. To see this, start from the
unconstrained formulation of Lagrange, and apply the above inequality, with $\phi = L$ the Lagrangian
of the original problem, and $y = \lambda$ the Lagrange vector multiplier.

\section{Nonholonomic Lagrange and Wolfe \\dual programs}

The Pfaff nonholonomic constraints in optimal programs were introduced by the
mathematical school coordinated by Prof. Dr. Constantin Udri\c ste at University
Politehnica of Bucharest (see \cite{DD}, \cite{DT}, \cite{RUU}-\cite{UFO}).

In this section we address the following original issues:
(i) difference between a program constrained by an integral submanifold of a Pfaff equation
and a program constrained by a Pfaff equation;
(ii) the non-holonomic Lagrange-dual problem with weak respectively strong duality;
(iii) the non-holonomic Wolfe-dual problem; (iv) pertinent examples.

Let $$\omega= \omega_i(x)dx^i= 0,\,i=1,...,n,\,\,x \in R^n,\,\, n\geq 3,$$ be a
non-completely integrable Pfaff equation (see also, \cite{D}, \cite{E}, \cite{M}, \cite{S}). The condition $n\geq 3$ is imposed by the nonholonomy theory.
For theoretical reasons, we understand that the co-vector field $\omega =(\omega_i(x))$ is $C^1$ on $R^n$, and has
no critical point in $R^n$. To this Pfaff equation we attach the $(n-1)$-hyperplane
$$H_x= \{\omega_i(x)dx^i= 0\}=\{y=(y^1,...,y^n)\in R^n\,|\, \omega_i(x)y^i=0\}$$
in the $n$-space $R^n$.

\begin{definition}
(i) A submersion $g=(g_\alpha),\,\alpha=1,...,n-p,$ is called solution of the Pfaff equation $\omega= \omega_i(x)dx^i= 0$
if $g_1(x)=0,...,g_{n-p}(x)=0$ and $dg_1(x)=0,...,dg_{n-p}(x)=0$ imply $\omega= \omega_i(x)dx^i= 0$.

(ii) A $p$-dimensional submersed submanifold $$(M,\omega, g, =)$$ of $R^n$ is called an integral manifold of the Pfaff equation on $R^n$ if
$$dg(M_x)\subseteq H_x,\,\, \hbox{for each point}\, x\, \hbox{in}\, M.$$
\end{definition}

\begin{definition}
The Pfaff equation is said to be {\it completely integrable} if there is one and only one integral
manifold of maximum possible dimension $n-1$ through each point of $R^n$.
\end{definition}

\begin{theorem} {\bf (Frobenius theorem)} A necessary and sufficient condition for
the Pfaff equation $\omega=0$ to be completely integrable is $\omega \wedge d\omega=0$.
\end{theorem}
Here $d\omega$ is the differential form of degree $2$ obtained from $\omega$ by exterior differentiation,
and $\wedge$ is the exterior product.

\subsection{Integral curves}

Let $(\Gamma,g=(g_1,...,g_{n-1}), =)$ be an integral curve of the non-completely
integrable Pfaff equation $\omega=\omega_i(x)dx^i=0$ can be written in the Cartesian implicit form
$$\Gamma: g_1(x)=0,..., g_{n-1}(x)=0,$$
i.e., the equations
$g_1(x)=0,..., g_{n-1}(x)=0$ and $dg_1(x)=0,..., dg_{n-1}=0$ imply
$\omega_i(x)dx^i=0$, for any $dx$. This means that
$$\det\left(\begin{array}{ccc}\frac{\partial g_1}{\partial x^1}&...&\frac{\partial g_1}{\partial x^n}\\ \
...&...&...\\ \
\frac{\partial g_{n-1}}{\partial x^1}&...&\frac{\partial g_{n-1}}{\partial x^n}\\ \
\omega_1&...&\omega_n \end{array}\right) =0.$$
It follows that there exist the functions $\nu_1(x),...,\nu_{n-1}(x)$, and the constant $\mu$, such that
$$\sum_{\alpha=1}^{n-1}\nu_\alpha(x) \frac{\partial g_\alpha}{\partial x^i}(x)= \mu \omega_i(x),\,\, i=1,...,n,$$
for any $x$.

\subsection{The connection between critical points on \\an integral submanifold and critical points \\with Pfaff non-holonomic constraint}

Let $f:R^n\to R$ be a $C^2$ function. Let us consider a non-holonomic program:
$$\min f(x)\,\, \hbox{subject to}\,\, \omega = \omega_i(x)dx^i=0,\,\, i=1,...,n.\leqno(NP)$$

Suppose that a $p$-dimensional integral submanifold of the Pfaff equation $\omega = \omega_i(x)dx^i=0$ is
$(M,\omega, g=(g_1,...,g_{n-p}), =)$. We consider the attached program
$$\min_x f(x)\,\,\hbox{subject to}\,\, x \in M$$
or, equivalently,
$$\min_x f(x)\,\,\hbox{subject to}\,\, g_1(x)=0,...,g_{n-p}(x)=0.$$

\begin{theorem}
 A point in $R^n$ is a nonholonomic constrained critical point if and only if it is
 critical point constrained by each integral submanifold containing it.
\end{theorem}

{\bf Proof} ({\bf if}) For a given integral submanifold, the associated Lagrange function is
$$L(x,\lambda)= f(x)+\lambda_1 g_1(x)+...+ \lambda_{n-p} g_{n-p}(x).$$
The critical point conditions are
$$df(x)+ \lambda_1 dg_1(x)+...+ \lambda_{n-p} dg_{n-p}(x)=0,\,\,\forall dx; g_1(x)=0,...,g_{n-p}(x)=0.$$

Let $X$ be a vector field tangent to the integral submanifold (and hence from distribution), i.e.,
$$dg_1(X)=0,...,dg_{n-p}(X)=0,\,\omega(X)=0.$$
It follows that at a critical point we must have $df(X)=0$.

If the integral submanifold is arbitrary (both as dimension and as way of description), i.e., $g$ is arbitrary, i.e., $X$ is an arbitrary
vector field tangent to $M$, then from the relations
$\omega(X)=0, df(X)=0, \forall X$, we obtain the existence of a constant $\mu$ such that,
at a critical point, which is independent on $g$, we must have
$$df(x)+ \mu \omega(x)=0,\,\,\forall dx.\leqno(NCP)$$

({\bf only if}) Each nonholonomic constrained critical point belongs to each critical point set
associated to a constrained integral submanifold. Indeed, if $X$ is a vector field tangent to an integral submanifold,
then
 $$dg_1(X)=0,...,dg_{n-p}(X)=0,\,\omega(X)=0.$$
 It follows the existence of multipliers $\lambda_1,...,\lambda_{n-p}$ such that
 $$\omega = \lambda_1 dg_1(x)+...+ \lambda_{n-p} dg_{n-p}(x).$$
 But $df= -\mu \omega$.

\subsubsection{Attached Riemannian geometry}
Let us consider the system "NCP = nonholonomic critical point", where the parameter $\mu$ is arbitrary.
According to the implicit function theorem, if the matrix
$$\left(a_{ij}\right)=\left(\frac{\partial^2 f}{\partial x^i \partial x^j}+\mu \frac{\partial \omega_i}{\partial x^j}\right)$$
is non-degenerate at a fixed critical point, then the system define a curve $x=x(\mu)$. On the other hand, the matrix of elements
$g_{jk}=\delta^{il} a_{ij}a_{kl}$ is a Riemannian metric. Symbolically, $a= (a_{ij})$, $g= {}^ta\,a$, $g^{-1} = a^{-1}\,{}^t(a^{-1})$
 and the geometry induced by $g$ follows by usual rules.

By differentiation with respect to $\mu$, we obtain
$$a_{ij}\frac{dx^j}{d\mu} +\omega_i=0.$$
Let $\eta_k= \delta^{il} a_{kl}\omega_i$. Then $g_{kj}\frac{dx^j}{d\mu}+ \eta_k=0$
and hence
$$g_{kj}\frac{dx^k}{d\mu}\frac{dx^j}{d\mu}+ \eta_k \frac{dx^k}{d\mu}=0$$

{\bf Proposition} {\it The angle between the vectors $\eta_k$ and $\frac{dx^j}{d\mu}$ is always obtuse}.

{\bf Completion} It follows
$$\left(\frac{\partial^2 f}{\partial x^i \partial x^j}+\frac{\mu}{2} \left(\frac{\partial \omega_i}{\partial x^j}+ \frac{\partial \omega_j}{\partial x^i}\right)\right)\frac{dx^i}{d\mu} \frac{dx^j}{d\mu} +\omega_i \frac{dx^i}{d\mu}=0.$$
We reinterpret this equality, introducing the {\it fundamental tensor}
$$h_{ij}=\frac{\partial^2 f}{\partial x^i \partial x^j}+\frac{\mu}{2} \left(\frac{\partial \omega_i}{\partial x^j}+ \frac{\partial \omega_j}{\partial x^i}\right),$$
and writing
$$h_{ij}\frac{dx^i}{d\mu} \frac{dx^j}{d\mu} +\omega_i \frac{dx^i}{d\mu}=0.$$

\subsection{Eliminating ambiguities by \\geometric interpretation}

The Pfaff equations theory is a source of misunderstanding for beginners.
We can eliminate such problems thinking in terms of differential geometry.

\subsubsection{Language of Vranceanu}

Let $\gamma_{x_0}$ be the image of an integral curve of the
Pfaff equation $\omega_i(x)dx^i= 0$ through the point $x_0$. Denote by $\Sigma_{x_0}=\{\gamma_{x_0}\}$
the family of all images of integral curves through the point $x_0$. The pair
$$(D,\Sigma),\,\, \Sigma=\{\Sigma_{x_0}\,|\, x_0\in D\}$$
is called {\it nonholonomic hypersurface} on $D$ attached to the Pfaff equation $\omega_i(x)dx^i= 0$ (see also, \cite{Vr}).

To the Pfaff equation $\omega_i(x)dx^i= 0$ and to the point $x_0\in D$, we attach the unique $(n-1)$-hyperplane
$$H_{x_0}= \{x \in R^n\,|\, \,\omega_i(x_0)(x^i - x_0^i)=0\}.$$
Since all straight lines tangent to integral curves which pass through $x_0$ are included in $H_{x_0}$,
the hyperplane $H_{x_0}$ is called the {\it tangent hyperplane} at $x_0$ of $(D,\Sigma)$.


\subsubsection{Language of distributions}

Let us introduce the $(n-1)$-hyperplane
$$H_x= \{\omega_i(x)dx^i= 0\}=\{y=(y^1,...,y^n)\in R^n\,|\, \omega_i(x)y^i=0\}$$
in the $n$-space $R^n$. The rule $x\to H_x$ gives a field $H$ of hyperplanes in $R^n$,
or what we call {\it $(n-1)$-dimensional distribution: a linear subbundle of the tangent bundle}.
In short $H=\cup_{x\in R^n}H_x\subset TR^n$.
In subsequent explanations, we shall prefer the distributions language being more suggestive.

Similarly, we can introduce the half-hyperplane
$$H_x^-= \{\omega_i(x)dx^i\leq 0\}=\{y=(y^1,...,y^n)\in R^n\,|\, \omega_i(x)y^i\leq 0\}.$$
The rule $x\to H^-_x$ gives a field $H^-$ called {\it $(-)$-distribution}.

\begin{definition}
A $p$-dimensional submersed submanifold $(M,H, g, =)$ of $R^n$ is called an integral manifold of the distribution $H$
on $R^n$ if
$$dg(M_x)\subseteq H(g(x))=H_x,\,\, \hbox{for each point}\, x\, \hbox{in}\, M.$$
\end{definition}

The field $H^-$ is used when the program refers to {\it manifolds whose boundary contains an integral manifold of a Pfaff equation}.
A manifold with boundary is a manifold with an edge. The boundary of a $(p+1)$-manifold with boundary is a $p$-manifold.
In technical language, a manifold with boundary is a space containing both interior points and boundary points.

\subsection{The Lagrange dual problem}

Let $(M,\omega, g=(g_1,...,g_{n-p}), =)$ be a $p$-dimensional integral submanifold
of the Pfaff equation $\omega = \omega_i(x)dx^i=0$ and $f:R^n\to R$ be a $C^2$ function.
Denote $g=(g_\alpha),\,\alpha=1,...,n-p,$ and we introduce the set
$$\Omega = \{x\in R^n \,|\,g_\alpha(x)= 0,\,\alpha=1,...,n-p\}= \{x\in R^n \,|\,g(x)= 0\}.$$
For each program
$$\min_x f(x)\,\,\hbox{subject to}\,\, x \in M,$$
equivalently
$$\min_x\{f(x)\,|\, x \in \Omega\},$$
we can repeat the theory of dual programs.
The {\it Lagrange function (or Lagrangian)} of this program is
$$L(x, \lambda) = f(x)+ \sum_{\alpha=1}^{n-p} \lambda_\alpha g_\alpha(x)
= f(x)+<\lambda,g>, x\in R^n, \lambda\in R^{n-p}.$$
The critical points with respect to the variable $x$ are given by the system
$$df(x)+ <\lambda,dg(x)>=0, \forall dx.\leqno(1)$$
It follows $x=x(\lambda)$, the dual function $\psi(\lambda)= L(x(\lambda), \lambda)$ and the {\it Lagrange dual problem}
$$\max_\lambda \psi(\lambda).$$

But, what we understand by the dual theory
for the nonholonomic program (NP)? Of course, we must ask an arbitrary integral submanifold.
That is why, the system (1) must be replace with the system
$$df(x)+ \mu \omega(x) =0, \forall dx.$$
The solution of this system is of the form $x=x(\mu)$, an arbitrary dual function is
$\psi(\mu,\lambda)= L(x(\mu), \lambda)$ and the {\it Lagrange dual problem} can be written
$$\max_{\mu,\,\lambda} \psi(\mu,\lambda).$$

Each Lagrange function $L(x, \lambda) $ of (NP)
is linear in $\lambda$.

If the distribution $H$ is described by $q$ Pfaff equations, then $\lambda$ has $(n-p)q$ components and
$\mu$ has $q$ components.

\subsubsection{Dual nonholonomic Lagrange function}

In the non-holonomic context, the constraint function $g$ does not exists, but it should be built
at least on the critical point set $\{x(\mu)\}$.

By analogy with the holonomic equality $g( \bar x(\mu )) = c$,
from the relation ${\displaystyle \frac {df_*}{d \mu} \left ( \frac {dc}{d \mu} \right )^{-1} = - \mu} $,
 we can define $c(\mu ) $ by Cauchy problem
$$ \frac {dc}{d \mu } = - \frac {1}{\mu }\, \frac {df_*}{d \mu}\,,\,\,\,\,\,\, c(\mu_0) = 0\,, \leqno (EC)$$
and then a Lagrange dual function will be
$$  \theta (\mu ) = f_*(\bar x(\mu ) ) + \mu \, c(\mu )\,. \leqno (LDF)$$
This function has the derivative $\theta^\prime (\mu )= c(\mu )$.

{\bf Example 1.} Let the objective function be $f(x,y,z) = x^2 + y^2 + z^2$ and the constraint Pfaff form $\omega = x dy + dz = 0$.
Then the critical points condition $df + \mu \omega = 0$ gives us
$$ \bar x(\mu ) =0\,,\,\, \bar y(\mu ) = 0\,,\,\, \bar z(\mu ) = - \frac {\mu }{2}\,,\,\, f_*(\bar x,\bar y, \bar z) = \frac {\mu^2}{4}\,.$$
If, for instance, we take $\mu_0 = 2$, the solution for the {\it primal problem} will be $\bar x=0, \bar y =0, \bar z = -1$ and $ f_*= 1$.

For the {\it Lagrange dual problem} the equation (EC) gives us $$\frac {dc}{d \mu } = - \frac {1}{\mu }\, \frac { \mu}{2} = - \frac {1}{2}\,.$$
Consequently ${\displaystyle c(\mu) = - \frac {\mu}{2} + \alpha }$ . If, as instance, $c_0 = 0$ for $\mu_0 = 2$, then $\alpha = 1$ and the
Lagrange dual function is $$\theta (\mu ) = \frac {\mu^2}{4} + \mu \left ( - \frac {\mu}{2} + 1 \right ) = - \frac {\mu^2}{4} + \mu\,.$$
It follows ${\displaystyle \theta ^\prime (\mu ) =  - \frac {\mu}{2} + 1} = 0$. Hence $\mu_0 = 2$ and we obtain the same solution as in primal
problem.

{\bf Example 2.} Let consider the function $f(x,y,z)=x^2+y^2-z$ and the Pfaff form $\omega =x dy - z dz$. Solve the {\it primal problem}
$$f(x,y,z)= extremum\,\,\, \hbox {with constraint}\,\,\, \omega = 0.$$
The differential Lagrange form of the  problem is
$$dL = 2x dx + 2y dy - dz + \mu (x dy - z dz)\,. $$
The condition $dL = 0$, (as differential form) leads to the equations $2x = 0,\, 2y + \mu x = 0,\, -1 - \mu z = 0$, whose solutions,
the critical points, are $x = 0,\, y = 0,\, z = -1/ \mu \, $. For $\mu < 0$ the critical points are points of constrained minimum and
the corresponding minimum values are $f_* = 1/ \mu \,.$ For $\mu \geq 0$ the critical points are not constrained extremum points.

For construct the {\it dual problem}, let us use the above described method. From equation (EC) we have $dc/ d\mu = 1/ \mu ^3$. Then
$\displaystyle c(\mu ) = - \frac {1}{2} \left ( \frac {1}{\mu ^2} - \frac {1}{\mu ^2_0} \right )$ and the Lagrange dual function will be
$$\theta (\mu) = \frac {1}{\mu } - \frac {\mu }{2} \left ( \frac {1}{\mu ^2} - \frac {1}{\mu ^2_0} \right )\,.$$
The critical points are given by the equation
$$\theta ^\prime (\mu) =  - \frac {1}{2} \left ( \frac {1}{\mu ^2} - \frac {1}{\mu ^2_0} \right ) = c(\mu ) = 0$$
and a solution is $\mu _0$, i.e. the strong duality holds.

{\it Another way } Remind that a point $(x_0,y_0,z_0)$ is a minimum (maximum) point for the function $f(x,y,z)$,
constrained by the Pfaff equation $\omega = 0$, if this point is a minimum (maximum) point for $f$ restricted at any line solution of $\omega = 0$,
passing through $(x_0,y_0,z_0)\,$. So we reformulate the {\it primal problem} for a suitable line passing through critical points, in our case $(0,0,1/c)\,$.
Such a line has the cartesian implicit equations
$$\left \{ \begin{array}{c} y^2 - z - 1/c = 0 \\ 2yz - x = 0\,. \end{array} \right .$$
Then the Lagrange function of the  problem is
$$L(x,y,z, \lambda , \mu ) = x^2+y^2-z + \lambda (y^2 - z - 1/c) + \mu (2yz - x)\,. $$
We obtain the following system of the critical points:
$$\frac {\partial L}{\partial x}=2x- \mu =0 $$
$$\frac {\partial L}{\partial y}=2y+2\lambda y + 2\mu z =0 $$
$$\frac {\partial L}{\partial z}=-1- \lambda + 2 \mu y=0\,. $$
This system has the solution $(0,0,1/c)\,$ for $\mu = 0 $ and $\lambda = 1\,.$

\subsection{Case of nonholonomic inequalities}

Let $(R^n,\omega, g=(g_1,...,g_{n-p}),\leq, \preceq)$ be a subset of $R^n$,
attached to the Pfaff inequation $\omega = \omega_i(x)dx^i\leq 0$, whose boundary
contains the $p$-dimensional integral submanifold $(M,\omega, g=(g_1,...,g_{n-p}),=, =)$ of the Pfaff equation.

Let $f:R^n\to R$ be a $C^2$ function.
Denote $g=(g_\alpha),\,\alpha=1,...,n-p,$ and we introduce the set
$$\Omega = \{x\in R^n \,|\,g_\alpha(x)\leq 0,\,\alpha=1,...,n-p\}= \{x\in R^n \,|\,g(x)\preceq 0\}.$$

\begin{theorem} ({\bf weak duality}) The dual function yields lower bounds of the initial optimal value $f_*$, i.e., for any
$\lambda$, we have $\varphi(\lambda)\leq f_*$. In other words,
$$\sup_\lambda\,\{\varphi(\lambda)\,|\, \lambda \succeq 0\}\leq \min_{x\in \Omega}\,\,\{f(x)+ <\lambda,g(x)>,\,\, x\in \Omega,\,\, \lambda\succeq 0\}.$$
\end{theorem}

\begin{theorem} {\bf (strong duality)} If the program (P) satisfies the Slater condition and has finite optimal
value, then
$$\sup_\lambda\,\{\varphi(\lambda)\,|\, \lambda \succeq 0\}= \min_{x\in \Omega}\{f(x)\}.$$
Moreover, then the dual optimal value is attained.
\end{theorem}

\subsection{Nonholonomic Wolfe dual}
The problem
$$\max_{x,\,\mu}\,\,\{f(x)\}$$
subject to
$$\frac{\partial f}{\partial x^i}(x) + \mu\, \omega_i(x)=0,\,\, \mu \geq 0$$
is called the {\it nonholonomic Wolfe dual} (WDNP) of the nonholonomic program (NP).

It follows $x=x(\mu)$ and $f(x(\mu))$. That is why, solving the dual problem
is equivalent to find extrema of the function $\mu \to f(x(\mu))$.

\subsection{Examples}

(1) see \cite{UFO}, p.190-191 {(\bf Consumer theory with nonholonomic constraint)}
Suppose that
$$u(x)=x_1^{\alpha_1}x_2^{\alpha_2}\cdots x_n^{\alpha_n},\,\,\alpha_i\geq 0,\,\,\sum_{i=1}^n\, \alpha_i<1,\,\, x=(x_1,...,x_n) \in R_+^n$$
is the utility function defined over $n$ goods. Denote by $p_i(x)>0,\, i=1,...,n,$ the prices of the goods.

Let us determine what is the proportion of income that the associated consumer will spend on each good, if the
budget constraint is the Pfaff inequality $\sum_{i=1}^n\, p_i(x)dx_i\geq 0$.

The utility function is concave. We must look for critical points of the
utility function $u$ subject to the given nonholonomic constraint. To determine the
constrained critical points, we use the Lagrange $1$-form
$$\eta = du(x) - \mu \sum_{i=1}^n\, p_i(x)dx_i= \sum_{i=1}^n\left(\frac{\alpha_i}{x_i}\,u(x) - \mu p_i(x)\right)dx_i,$$
and we write the system
$$\alpha_i u(x)- \mu p_i(x)x_i=0,\,i=1,...,n.$$

Suppose we have a critical point (solution) $x^* = x^*(\mu),\, \mu>0$.
Each component $X_i^*$ of the critical point is the quantity consumed of the $i^{th}$ good.
Moreover, $p_i(x^*)x_i^*$ is the income spend on the $i^{th}$ good and
$\sum_{i=1}^n p_i(x^*)x_i^*$ is the total income. Since
$$u(x)\sum_{i=1}^n \alpha_i - \mu \sum_{i=1}^n p_i(x^*)x_i^*=0,$$
we find the proportion of the income spent on the $i^{th}$ good,
$$\frac{p_i(x^*)x_i^*}{\sum_{i=1}^n p_i(x^*)x_i^*}=\frac{\alpha_i}{\sum_{i=1}^n\alpha_i},$$
which is independent on consumed quantities and of prices ({\it economic law}).

Suppose we are interested in the maximum of the utility function $u$ subject to the nonholonomic constraint.
To solve this problem, as usual we look for a critical point $x^* = x^*(\mu),\, \mu>0$,
which verify the equality in the budget constraint (hyperplane) $\sum_{i=1}^n p_i(x^*(\mu))dx_i=0$ and
the negative definiteness of the restriction of the quadratic form
$$d^2u(x^*) - \frac{\mu}{2}\,\,\sum_{i,j=1}^n \left(\frac{\partial p_i}{\partial x_j}+\frac{\partial p_j}{\partial x_i}\right)(x^*)\,dx_i\,dx_j$$
to the budget hyperplane.

(2) {\bf Nonholonomic initial program} Find extremum points of the function $f(x,y,z)=2xy + z^2$ subject to $zdx-dy \geq 0,\, xdy+dz \geq 0$\,\, (see \cite{UD6}).
The constrained critical points are solutions of the system
$$2y-\mu_1 z=0,\,\, 2x + \mu_1 -\mu_2 x =0,\,\, 2z-\mu_2=0,\,\,\mu_1 \leq  0,\,\,\mu_2\leq 0.$$
It follows the family of critical points
$$x= \frac{\mu_1}{\mu_2-2},\,\, y=\frac{1}{4}\mu_1 \mu_2, \,\,z=\frac{1}{2}\mu_2,\,\,\mu_1 \leq  0,\,\,\mu_2\leq 0.$$

The nature of each critical point is fixed by the signature of the quadratic form $(4-\mu_2)dxdy-\mu_1 dxdz+2dz^2$
restricted to $\frac{\mu_2}{2}dx-dy=0,\,\frac{\mu_1}{\mu_2-2}dy + dz =0$. It follows the restriction
$q= \mu_2\left(\frac{\mu_1^2 }{(\mu_2 -2)^2}+2 -\frac{\mu_2}{2}\right)\,dx^2$, which is negative definite.
All critical points are maximum points. The manifold of critical points has the implicit Cartesian equation
$$xz(z-1)=y,\,x\geq 0, y\geq 0,\, z\leq 0.$$
The maximum value of the function $f$ is
$$f(x(\mu_1,\mu_2), y(\mu_1,\mu_2), z(\mu_1,\mu_2))= \frac{\mu_1^2 \mu_2}{2(\mu_2-2)}+ \frac{1}{4}\mu_2^2.$$

If we change the constraints into
$$zdx-dy\leq 0,\, xdy+dz \leq 0,$$
then it will be sufficient to have positive multipliers and a positive
definite quadratic form $q$ in order that each critical point becomes a minimum point.

The PDEs system which gives us $c_1(\mu_1 , \mu ), c_2(\mu_1 , \mu_2 )$ is (see 1.2 )
$$\mu_1 \, \frac{\partial c_1}{\partial \mu_1 } + \mu_2 \, \frac{\partial c_1}{\partial \mu_2} = - \frac {\mu_1 \mu_2 }{\mu_2 - 2 }$$
$$\mu_1 \, \frac{\partial c_2}{\partial \mu_1 } + \mu_2 \, \frac{\partial c_2}{\partial \mu_2} = - \frac{1}{2} \, \mu_2 + \left ( \frac {\mu_1 }{\mu_2 - 2} \right )^2 \,, $$
with solutions, respectively,
$$c_1 (\mu_1 , \mu_2 ) = - \mu_1 - \frac {2\, \mu_1 }{\mu_2 } \, \ln |\mu_2 -2| + \alpha _1 \left ( \frac {\mu_1}{ \mu_2} \right ) \,,$$
$$c_2 (\mu_1 , \mu_2 ) = - \frac {1}{2} \, \mu_2 - \frac {\mu_1^2}{\mu_2^2} \, \left ( \frac {2}{\mu_2 -2} - \ln |\mu_2 -2| \right ) +
\alpha _2 \left ( \frac {\mu_1}{ \mu_2} \right )\,,$$
where $\alpha _1\,, \alpha _2 $ are arbitrary functions. The condition $\displaystyle \frac{\partial c_1}{\partial \mu_2} =
\frac{\partial c_{2}}{\partial \mu_1} $ is verified, for instance, if $\displaystyle \alpha_1 \left ( \frac {\mu_1}{\mu_2} \right ) = ct. = \alpha _1 $
and $\displaystyle \alpha _2 \left ( \frac {\mu_1}{\mu_2 } \right ) = - \frac {\mu_1^2}{\mu_2^2} + \alpha _2\,, \alpha _2 = ct. $
Finally, the Lagrange dual function $\theta (\mu_1, \mu_2) = f_* + \mu_1 c_1 + \mu_2 c_2 $ is
$$\theta (\mu_1, \mu_2) =  - \frac {\mu_1^2}{2} - \frac {1}{4} \mu_2^2 - \frac { \mu_1^2}{\mu_2} \ln |\mu_2 -2| +
\alpha _1 \mu_1 + \alpha _2  \mu_2 \,.$$

{\bf Nonholonomic Wolfe dual program} Find the extrema of the function
$$(\mu_1,\mu_2)\to \varphi(\mu_1,\mu_2)=f(x(\mu_1,\mu_2), y(\mu_1,\mu_2), z(\mu_1,\mu_2)),\,\,\mu_1 \leq 0,\,\,\mu_2 \leq 0.$$
Since
$$\varphi(\mu_1,\mu_2)= \frac{\mu_1^2 \mu_2}{2(\mu_2-2)}+ \frac{1}{4}\mu_2^2\geq 0,$$
in the conditions of the problem, the critical point $\mu_1=0,\mu_2=0$ is a minimum point.
Also, all the points of the form $(\mu_1 <0,\mu_2=0)$
are minimum points.

\section{Extrema in canonical coordinates}

The next Theorem, which gives the canonical Pfaff forms (and hence canonical nonholonomic constraints),
is named after Jean Gaston Darboux who established it as the solution of the Pfaff problem (see \cite{D}, \cite{S}, \cite{Vr}).

\begin{theorem} ({\bf Darboux Theorem})
Suppose that $\omega$ is a differential $1$-form on an $n$ dimensional manifold, such that $d\omega$ has constant rank $p$. If
$\omega \wedge (d\omega)^p=0$ everywhere,
then there is a local system of coordinates $x^1,...,x^p, y^1, ..., y^p$ in which
$$\omega = \sum_{i=1}^p x^i dy^i.$$
If, on the other hand, $\omega \wedge (d\omega)^p \neq 0$ everywhere,
then there is a local system of coordinates $x^1,...,x^p$, $y^1, ..., y^p, z$ in which
$$\omega = \sum_{i=1}^p x^i dy^i + dz\,\, \hbox{(contact form)}$$
or
$$\omega = \frac{1}{2}\left(\sum_{i=1}^p x^i dy^i- \sum_{i=1}^p y^i dx^i\right)+dz\,\, \hbox{(symmetric normal form)}.$$
\end{theorem}

From the normal form of Darboux, we see that the maximal integral manifolds are of
dimension $p$. For the contact form equation
$$\omega = \sum_{i=1}^p x^i dy^i + dz=0,$$
they are given by
$$z=f(y^1,...,y^{p}), \,\,x_1=- \frac{\partial f}{\partial y^1},...,x_p=- \frac{\partial f}{\partial y^p},$$
where $f$ is a $C^2$ arbitrary function.

\subsection{Case of even number of variables}

Let us find the extrema of a function $f(x,y), x=(x^i), y=(y^i), i=1,...,n$, subject to a nonholonomic constraint
written as Pfaff equation $\omega = x^1 dy^1 + ... + x^n dy^n =0$.

The constrained critical points are solutions of the system
$$\frac{\partial f}{\partial x^i} =0,\, \frac{\partial f}{\partial y^i}+\mu x^i =0.$$

\subsection{Case of odd number of variables}

Let us find the extrema of a function
$f(x,y,z), x=(x^i), y=(y^i), i=1,...,n$ subject to a nonholonomic constraint
$\omega = x^1 dy^1 + ... + x^n dy^n + dz=0$. The constrained
critical points are solutions of the system
$$\frac{\partial f}{\partial x^i} =0,\, \frac{\partial f}{\partial y^i} + \mu x^i=0,\,\frac{\partial f}{\partial z} + \mu =0.$$

To decide the type of a critical point $(x_0,y_0,z_0)$, we use the restriction of the quadratic form
$Q= d^2f(x_0,y_0,z_0)+ \mu \delta_{ij}dx^idy^j$ to the hyperplane $\delta_{ij} x^{i}_0 dy^j+dz=0$
and its signature. Since
$$d^2f= \frac{\partial^2 f}{\partial x^i\partial x^j}dx^i dx^j+ \frac{\partial^2 f}{\partial y^i\partial y^j}dy^i dy^j + 2 \frac{\partial^2 f}{\partial x^i\partial y^j}dx^i dy^j$$
$$+ 2\left(\frac{\partial^2 f}{\partial x^i\partial z}dx^i +\frac{\partial^2 f}{\partial y^i\partial z}dy^i\right)dz + \frac{\partial^2f}{\partial z^2}dz^2,$$   the restriction of $Q$ is
$$q= \frac{\partial^2 f}{\partial x^i\partial x^j}dx^i dx^j+\frac{\partial^2 f}{\partial y^i\partial y^j}dy^i dy^j + 2\frac{\partial^2 f}{\partial x^i\partial y^j}dx^i dy^j + \mu \delta_{ij}dx^idy^j$$
$$- 2\left(\frac{\partial^2 f}{\partial x^i\partial z}dx^i + \frac{\partial^2 f}{\partial y^i\partial z}dy^i\right) \delta_{kl} x^{k}_0 dy^l + \frac{\partial^2f}{\partial z^2}(\delta_{kl} x^{k}_0 dy^l )^2.$$

\subsubsection{Another point of view}

The general solution of the Pfaff equation $\omega = x^1 dy^1 + ... + x^n dy^n + dz=0$ is
$$z=\varphi(y),\,\,x=- \frac{\partial \varphi}{\partial y}(y),\,\, \hbox{where}\,\, \varphi\,\, \hbox{is arbitrary}.$$

Consequently the previous nonholonomic program can be written
$$\min f(x,y,z)\,\,\hbox{subject to}\,\, z=\varphi(y),\,\, x=- \frac{\partial \varphi}{\partial y}(y).$$
In this form, it is similar to a classical program, but the function $\varphi$ is arbitrary. For each $\varphi$,
we attach a Lagrangian
$$L(x,y,z,\lambda_1,\lambda_2)= f(x,y,z) +\lambda_1(z-\varphi(y))+\lambda_2\left(x+ \frac{\partial \varphi}{\partial y}(y)\right).$$
It follows the system which describes the critical points
$$\frac{\partial L}{\partial x}= \frac{\partial f}{\partial x} +\lambda_2=0,\,\, \frac{\partial L}{\partial z}= \frac{\partial f}{\partial z} +\lambda_1=0$$
$$\frac{\partial L}{\partial y}= \frac{\partial f}{\partial y} -\lambda_1 \frac{\partial \varphi}{\partial y}+\lambda_2 \frac{\partial^2 \varphi}{\partial y \partial y}=0$$
$$\frac{\partial L}{\partial \lambda_1}= z-\varphi(y)=0,\,\,\frac{\partial L}{\partial \lambda_2}= x+ \frac{\partial \varphi}{\partial y}(y)=0.$$
This means that the critical points with respect to $x$, $y$ and $z$ must verifies the constraints of the initial program.

In this context, it is very clear what means a Lagrange dual program. The dual function is
$$\psi(\lambda_1,\lambda_2)= L(x(\lambda_1,\lambda_2),y(\lambda_1,\lambda_2),z(\lambda_1,\lambda_2),\lambda_1,\lambda_2).$$

{\bf Example}  Find extremum points of the function
$f(x,y,z)=x+y+z+\frac{1}{2}(x^2+y^2+z^2)$ subject to $dz-xdy=0$.

The constrained critical points are solutions of the system
$$\frac{\partial f}{\partial x}=0, \frac{\partial f}{\partial y}-\mu x=0, \frac{\partial f}{\partial z}+\mu=0,$$
i.e., $x=1, y=1-\mu, z=1+\mu$.

The nature of critical points is determined by the signature of the restriction $q$ of
the quadratic form $Q= -(dx^2+dy^2+dz^2)-\mu dx dy$ to the plane $dz=dy$. It follows
$q= -(dx^2+2dy^2) - \mu dx dy$. This quadratic form is negative definite for $\mu^2 <8$, i.e.,
$\mu \in (-2\sqrt{2},2\sqrt{2})$. In this case, all critical points are maximum points.

The function $\varphi(\mu)=f(x(\mu),y(\mu),z(\mu))= \frac{3}{2}-\mu^2$ is increasing on $(-2\sqrt{2},0)$
and decreasing on $(0,2\sqrt{2})$. Also, $\inf \varphi(\mu)= \frac{3}{2}-8$.

\subsubsection{Passing to an even number of variables}

The point $(x,y,z)$ belong to the contact manifold defined by $\omega =xdy +dz=0$. Let $t$ be the number to be multiplied with
$1$-form $\omega$ to obtain a point of a symplectic $(2n+2)$-manifold described by the $1$-form $\eta=t\omega = tx dy + t dz$.
If we pass to the coordinates $P= (P^I)=(p,p_0), p=tx,p_0=t$, $Q= (Q^I)=(q,q_0), q=y, q_0=z$, then $\eta = PdQ$ and hence $d\eta=dP \wedge dQ$.

Having in mind the changing of the variables, the function $f(x,y,z)$ becomes $f\left(\frac{p}{p_0}, q, q_0\right)=\varphi (P,Q)$.
The constrained critical points are solutions of the system
$$\frac{\partial \varphi}{\partial P^I} =0,\, \frac{\partial \varphi}{\partial Q^i}+\mu x^i =0,\,\frac{\partial \varphi}{\partial z}+ \mu=0.$$

\subsubsection{Contact Hamiltonian}

Let $X= \dot x^i\frac{\partial}{\partial x^i}+\dot y^i\frac{\partial}{\partial y^i} + \dot z\frac{\partial}{\partial z},i=1,...,n$ be the contact vector field.
Let $K(x,y,z)$ be the contact Hamiltonian, which is defined by
$$K(x,y,z)= \omega(X),\,\, X\rfloor d\omega|_{\omega=0} =dK.$$
In case of $\omega = xdy +dz$, we have $d\omega|_{\omega=0}=-dx\wedge dy$. It follows the Hamiltonian
$K=x\dot y + \dot z$ and the contact flow
$$\dot x= -\frac{\partial K}{\partial y} + x \frac{\partial K}{\partial z},\,\,\dot y= \frac{\partial K}{\partial x},\,\, \dot z = K-x \frac{\partial K}{\partial x}.$$

{\bf Open problem} Find extrema of the contact Hamiltonian $K(x,y,z)$
constrained by $\omega = xdy +dz=0$.


Authors' Address:\\

Constantin Udri\c ste, M\u ad\u alina Constantinescu, Ionel \c Tevy, Oltin Dogaru\\
University Politehnica of Bucharest\\
Faculty of Applied Sciences\\
Department of Mathematics and Informatics\\
Splaiul Independentei 313\\
Bucharest 060042, Romania\\
Email: udriste@mathem.pub.ro; vascatevy@yahoo.fr


\begin{thebibliography}{99}

\bibitem{B} H. Bonnel, {\it Analyse Fonctionnelle}, Maitrise de Math\' ematiques, Ing\' enerie Math\' eematiques, Universit\' e de La Reunion,
Facult\' e de Sciences et Technologies, 2015.

\bibitem{D} G. Darboux, {\it Sur le probl\` eme de Pfaff}, Bull. Sci. Math. 6 (1882), 14-36, 49-68.

\bibitem{H} R. B. Holmes, {\it Geometric Functional Analysis and Its Applications}, Springer-Verlag, New York, 1975.

\bibitem{R} K. Roos, {\it Nonlinear Programming}, LNMB Course, De Uithof, Utrecht, TUDelft, February 6 - May 8, A.D. 2006.

\bibitem{S} S. Sternberg, {\it Lectures on Differential Geometry}, Prentice Hall, 1964.

\bibitem{DD} O. Dogaru, V. Dogaru, {\em Extrema Constrained by $C^k$ Curves}, Balkan Journal of Geometry and Its Applications, 4, 1 (1999), 45-42.

\bibitem{DT} O. Dogaru, I. \c Tevy, {\em Extrema Constrained by a Family of Curves}, Proc. Workshop ob Global Analysis, Diff. Geom. and Lie Algebras, 1996, Ed. Gr. Tsagas, Geometry Balkan Press, 1999, 185-195.

\bibitem{E} I. Ekeland, {\em Exterior Differential Calculus and Applications to Economic Theory}, Quaderni Scuola Normale Superiore di Pisa, 1998, Italy.

\bibitem{M} R. Montgomery, {\em A Tour of Subriemanniene Geometries, Their Geodesics and Applications}, Mathematical Surveys and Monographs, 91, American mathematical Society, 2002.

\bibitem{RUU} V. Radcenco, C. Udri\c ste, D. Udri\c ste, {\em Thermodynamic Systems and Their Interaction}, Sci. Bull. P.I.B., Electrical Engineering, vol. 53, no. 3-4 (1991), 285-294.

\bibitem{TU} Gr. Tsagas, C. Udri\c ste, {\em Vector Fields and Their Applications}, Geometry Balkan Press, Bucharest, 2002.

\bibitem{UD1} C. Udri\c ste, O. Dogaru, {\em Mathematical Programming Problems with Nonholonomic Constraints}, Seminarul de Mecanic\u a, Univ. of Timi\c soara, Facultatea de \c Stiin\c te ale Naturii, vol. 14, 1988.

\bibitem{UD2} C. Udri\c ste, O. Dogaru, {\em Extrema with Nonholonomic Constraints}, Sci. Bull., Polytechnic Institute of Bucharest, Seria Energetic\u a, Tomul L, 1988, 3-8.

\bibitem{UD3} C. Udri\c ste, O. Dogaru, {\em Extreme condi\c tionate pe orbite}, Sci. Bull., 51 (1991), 3-9.

\bibitem{UD4} C. Udri\c ste, O. Dogaru, {\em Convex Nonholonomic Hypersurfaces}, Math. Heritage of C.F. Gauss, 769-784, Ed. G. Rassias, World Scientific, 1991.

\bibitem{UD5} C. Udri\c ste, O. Dogaru, I. \c Tevy, {\em Sufficient Conditions for Extremum on Differentiable Manifolds}, Sci. Bull., Polytechnic Institute of Bucharest, Electrical Engineering, vol. 53, no. 3-4 (1991), 341-344.

\bibitem{UD6} C. Udri\c ste, O. Dogaru, I. \c Tevy, {\em Extremum Points Associated with Pfaff Forms}, Presented at the 90th Anniversary Conference of Akitsugu Kawaguchi's Birth, Bucharest, Aug. 24-29, 1992; Tensor, N.S., Vol. 54 (1993), 115-121.

\bibitem{UD7} C. Udri\c ste, O. Dogaru, I. \c Tevy, {\em Open Problem in Extrema Theory}, Sci. Bull. P.U.B., Series A, Vol. 55, no.3-4 (1993), 273-277.

\bibitem{UD8} O. Dogaru, I. \c Tevy, C. Udri\c ste, {\em Extrema Constrained by a Family of Curves and Local Extrema}, JOTA, vol. 97, no.3, June 1998, 605-621.

\bibitem{UD9} C. Udri\c ste, O. Dogaru, I. \c Tevy, {\em Extrema Constrained by a Pfaff System}, Hadronic J. Supplement, USA, 1991-Proc. Int. Workshop on Fundam. Open Problems in Math., Phys. and Other Sciences, Beijing, August 28, 1997.

\bibitem{UD9} C. Udri\c ste, I. \c Tevy, M. Ferrara, {\em Nonholonomic Economic Systems}, see [28], 139-150.

\bibitem{UD10} C. Udri\c ste, I. \c Tevy, {\em Geometry of test Functions and Pfaff Equations}, see [28], 151-165.

\bibitem{UD11} C. Udri\c ste, O. Dogaru, I. \c Tevy, {\em Extrema with Nonholonomic Constraints}, Geometry Balkan Press, Bucharest, 2002.

\bibitem{UD12} C. Udri\c ste, O. Dogaru, M. Ferrara, I. \c Tevy, {\em Pfaff Inequalities and Semi-curves in Optimum Problems}, Recent Advances in Optimization, pp. 191-202, Proceedings of the Workshop held in Varese, Italy, June 13/14th 2002, Ed. G.P. Crespi, A. Guerraggio, E. Miglierina, M. Rocca, DATANOVA, 2003.

\bibitem{UD13} C. Udri\c ste, O. Dogaru, M. Ferrara, I. \c Tevy, {\em Pfaff inequalities and semi-curves in optimum problems}, in Edt. G. P. Crespi, A. Guerraggio, E. Miglierina, M. Rocca, Recent Advances in Optimization, Proceedings of the Workshop held in Varese, Italy, June 13-14, 2002, pp. 191-202.

\bibitem{UD14} C. Udri\c ste, O. Dogaru, M. Ferrara, I. \c Tevy, {\em Extrema with constraints on points and/or velocities}, Balkan Journal of Geometry and Its Applications, 8, 1(2003), 115-123.

 \bibitem{UD15} C. Udri\c ste, O. Dogaru, M. Ferrara, I. \c Tevy, {\em Nonholonomic optimization theory}, see [28], 177-192, Geometry Balkan Press, Bucharest, 2004.

 \bibitem{UFO} C. Udriste, M. Ferrara, D. Opris, {\it Economic Geometric Dynamics}, Geometry Balkan Press, Bucharest, 2004.

\bibitem{Vr} Gh. Vr\u anceanu, {\em Le\c cons de Geometry Differentielle}, Editions de l'Academie Roumaine, Bucarest, 1957-1975.


\end{thebibliography}
\end{document}